\documentclass[reqno]{amsart}

\usepackage{amsmath}
\usepackage{amssymb}
\usepackage{amsfonts}
\usepackage{graphicx}
\usepackage{amsthm}
\usepackage{enumerate}
\usepackage{lscape}
\usepackage{dsfont}
\usepackage{color}
\usepackage{mathtools}

\usepackage{setspace}
\newcommand{\CP}{\mathds{C}\mathrm{P}}

\newcommand{\C}{\mathds{C}}

\newcommand{\K}{K\"{a}hler}

\newcommand{\Hol}{{\operatorname{Hol}}}

\def\b{\beta}

\def\b1{{\rm id}}

\newtheorem{theor}{Theorem}[section]
\newtheorem{prop}[theor]{Proposition}

\newtheorem{lem}[theor]{Lemma}

\newtheorem{remark}{Remark}

\begin{document}

\title[Partially regular and cscK metrics]{Partially regular and cscK  metrics}

\author{Andrea Loi}
\address{(Andrea Loi) Dipartimento di Matematica e Informatica \\
         Universit\`a di Cagliari (Italy)}
         \email{loi@unica.it}

\author{Fabio Zuddas}
\address{(Fabio Zuddas) Dipartimento di Matematica e Informatica \\
         Universit\`a di Cagliari (Italy)}
         \email{fabio.zuddas@unica.it}

\thanks{
The first two  authors were  supported  by Prin 2015 -- Real and Complex Manifolds; Geometry, Topology and Harmonic Analysis -- Italy, by INdAM. GNSAGA - Gruppo Nazionale per le Strutture Algebriche, Geometriche e le loro Applicazioni and by GESTA - Funded by Fondazione di Sardegna and Regione Autonoma della Sardegna. }
\subjclass[2000]{53D05;  53C55; 53D45} 
\keywords{partial bergman kernel; balanced metric; regular metric; constant scalar curvature metric.}

\begin{abstract}
A \K\ metric $g$ with  integral \K\ form   is said to be  partially regular if  the partial Bergman kernel associated to $mg$ 
is a positive constant for all  integer $m$ sufficiently large. The aim of this paper is to prove that for all $n\geq 2$ there exists an $n$-complex dimensional 
manifold equipped with strictly partially regular and cscK metric $g$. Further, for $n\geq 3$, the   (constant) scalar curvature of $g$ can be chosen to be zero, positive or negative.

\end{abstract}
 
\maketitle

\tableofcontents  

\section{Introduction}

Let $M$ be an $n$-dimensional complex manifold endowed with a K\"ahler metric $g$. Assume that there exists a 
holomorphic line bundle  $L$ over $M$ such that   $c_1(L)=[ \omega ]_{dR}$,
where $\omega$ is the K\"{a}hler form associated to $g$ and $c_1(L)$
denotes the first Chern class of $L$
(such an $L$ exists if and only if  $\omega$ is an integral form).
Let $h$ be an  Hermitian metric on
$L$ such that its Ricci curvature ${\rm Ric}
(h)=\omega$. 
Here ${\rm Ric} (h)$ is the two--form on $M$ whose
local expression is given by
\begin{equation}\label{rich}
{\rm Ric} (h)=-\frac{i}{2\pi}
\partial\bar\partial\log h(\sigma (x), \sigma (x)),
\end{equation}
for a trivializing holomorphic section $\sigma :U\rightarrow
L\setminus\{0\}$. 
Consider the separable complex Hilbert space
$\mathcal{H}$ consisting of global holomorphic sections  $s$ of $L$
such that
$$\langle s, s\rangle=
\int_Mh(s(x), s(x))\frac{\omega^n}{n!}<\infty .$$
Assume ${\mathcal H}\neq \{0\}$.  Let ${\mathcal S}\subseteq {\mathcal H}$
be a complex subspace of  ${\mathcal H}$ and let 
$s_j$, $j=0,\dots ,N$ ($\dim {\mathcal S}=N+1\leq\infty$) be an orthonormal basis of $\mathcal{S}$. 
In  this paper  we say that  the metric  $g$ is a {\em partially balanced metric with respect to ${\mathcal S}$}
if the smooth function, called {\em partial Bergman kernel},
$$T^{\mathcal S}_g(x)=\sum_{j=0}^Nh(s_{j}(x), s_{j}(x))$$
is a positive constant   ($T^{\mathcal S}_g$ really depends only on the metric $g$ and not on the orthonormal basis chosen). When ${\mathcal S}={\mathcal H}$ then
 $T^{\mathcal H}_g=T_g$ is Rawnsley's epsilon function (see \cite{rawnsley}, \cite{cgr1},  \cite{LZeddaCH} and references therein) and being  $g$ a partially balanced metric with respect to ${\mathcal H}$ means that $g$ is balanced in Donaldson's terminology (see \cite{donaldson} and \cite{hombal} for the compact case and \cite{arlquant}  for the noncompact case). Obviously, if $M$ is compact, $\mathcal{H}=H^0(L)$, where 
$H^0(L)$ is the (finite dimensional) space of global holomorphic sections of $L$.
In the sequel we will say that a \K\ metric $g$ on a complex manifold $M$ is {\em strictly} partially balanced with respect to ${\mathcal S}$ if 
$T_g^{\mathcal S}$ is a positive constant for  ${\mathcal S}$ {\em strictly} contained in ${\mathcal H}$.
Notice that given  a (strictly) partially balanced metric $g$ on a complex manifold $M$ with respect to  
${\mathcal S}\subseteq {\mathcal H}$ then, for all $x\in M$ there exists $s\in {\mathcal H}$ not vanishing at $x$ (the so called free based point condition in the compact case).   Then the Kodaira's map  $\varphi :M\rightarrow \C P^{N}, x\rightarrow [s_0(x):\cdots :s_N(x)]$ is well defined.
Morever, it is not hard to see that $\varphi^*\omega_{FS}=\omega+\frac{i}{2\pi}\partial\bar\partial\log T^{\mathcal S}_g$  and hence in the partially balanced case $\varphi$ is indeed a \K\ immersion, i.e. $\varphi^*g_{FS}=g$ where $g_{FS}$ (resp. $\omega_{FS}$) 
 is the Fubini-Study metric (resp. form) on $\C P^{N}$. 
Thus, by using the celebrated Calabi's rigidity theorem (\cite{Cal}, \cite{LoiZedda-book}) one  deduces  that in
the definition of partially balanced metric  the space ${\mathcal S}$ is determined up to unitary transformations of $\mathcal{H}$ and one can then simply   speak of  strictly partially balanced metric without specifying the space ${\mathcal S}$.
Partial Bergman kernels and their asymptotics  have been recently considered, when ${\mathcal S}$ is the subspace of ${\mathcal H}$ consisting of those holomorphic sections of $L$ vanishing  at a prescribed  order on an analytic subvariety of $M$ (see \cite{pokross}, \cite{rosssinger}, \cite{ZePeng}, \cite{ZePeng2}, \cite{ZePeng3}, \cite{ZePeng4}). Notice that  our definition is more general, since we are not fixing any analytic subvariety of $M$.

In  this paper we address the study of those metrics $g$ such that $mg$ is  strictly partially  balanced (with respect to some complex subspace ${\mathcal S}_m\subset{\mathcal H}_m$) for $m$ sufficiently large, were
${\mathcal H}_m$ denotes the Hilbert space of global holomorphic sections of  $L^m$ (the $m$-th tensor power of $L$) such that $\langle s, s\rangle_m=
\int_Mh_m(s(x), s(x))\frac{\omega^n}{n!}<\infty$ and $h_m$ is the Hermitian metric on $L^m$ such that 
${\rm Ric} (h_m)=m\omega$. Throughout the paper  a metric satisfying the previous condition  will be called a {\em strictly partially  regular metric}.  When ${\mathcal S}_m={\mathcal H}_m$
a  partially  regular metric $g$ is regular as defined in  \cite{loireg} (see also \cite{ALZ} and references therein) and it follows  that $g$ is a cscK (constant scalar curvature \K) metric.

Therefore it seems natural to address the following:

\vskip 0.3cm

\noindent
{\bf Question:}
Does there exist a complex manifold $M$ equipped with a   cscK metric $g$ 
such that  $g$ is strictly partially  regular?

\vskip 0.3cm

The aim of this paper is to provide a positive answer to the previous question in the noncompact case
as expressed by the following theorem proved in the next section.

 \begin{theor}\label{mainteor}
For all positive integer $n\geq 2$ there exist an $n$-dimensional noncompact complex manifold $M$ equipped with a 
strictly partially regular cscK metric  $g$. Furthermore for $n\geq 3$  the scalar curvature of $g$ can be chosen to be zero, positive or negative. 
\end{theor}

As we have already noticed above  a partially balanced metric $g$ on a complex manifold $M$ is authomatically  projectively induced. 
On the other hand  in the complex one-dimensional case a projectively induced cscK metric is regular, being homogeneous (actually a complex space form) and so it cannot be strictly partially regular.
This is the reason why  in Theorem \ref{mainteor} we assume  $n\geq 2$.

Notice also that it is conjecturally true that  a projectively induced  cscK metric $g$ on a {\em compact} complex manifold $M$ is homogeneous and hence regular (\cite{DHL}, \cite{LZedda11}, see also \cite{FCAL} for an example of regular non homogeneous complete \K\ metric on the blow-up of $\C^2$ at the origin). Hence we believe that the previous question has a negative answer in the compact case.

Finally we still  do not know if there exist complex $2$-dimensional manifolds admitting  strictly partially  regular cscK metrics with non negative scalar curvature (see the proof of Theorem \ref{mainteor}). 

\vskip 0.3cm
\noindent {\bf Acknowledgments}. The authors would like to thank Michela Zedda for her useful comments on the paper.

\section{Proof of Theorem \ref{mainteor}}
In order to prove Theorem \ref{mainteor} we first show that the punctured unit  disk 
$${\mathcal D}_*=\{z=(z_1, z_2)\in\C^2 \ |\ 0< \ |z|^2<1\}, \ |z|^2=|z_1|^2+|z_2|^2,$$
can  be equipped with a strictly partially regular cscK metric $g_*$ (see Proposition \ref{propmain} below)
whose associated \K\ form is given by:
$$\omega_*=\frac{i}{2\pi}\partial\bar\partial\Phi_*, \ \Phi_*(z)=\log |z|^2-\log (1 -  |z|^6).$$
A direct computation shows that 
its volume form is given by 
\begin{equation}
\frac{\omega^2}{2!}=\frac{9r(1+2r^3)}{(1-r^3)^3}(\frac{i}{2\pi})^2dz_1\wedge d\bar z_1\wedge dz_2\wedge d\bar z_2, \ r=|z_1|^2+|z_2|^2,
\end{equation}
from which one easily sees (cfr. \cite{LMZszego}) that $g_*$ is a cscK (not Einstein) metric  with constant  scalar curvature 
\begin{equation}\label{sg*}
s_{g_*}=-24\pi.
\end{equation}

\begin{remark}\rm 
Notice that the metric $g_*$  is not complete at the origin while 
it is complete at  $\{z\in \C^2 \ | \ |z|^2=1\}$.
\end{remark}

Let $m$ be a positive integer and  ${\mathcal D}_*\times\C$ be the trivial holomorphic line bundle on ${\mathcal D}_*$ endowed with
the hermitian metric  
\begin{equation}\label{h*}
h_*^m(z, \xi)=e^{-m\Phi_*(z)}|\xi|^2=\frac{(1-|z|^6)^{m}|\xi|^2}{|z|^{2m}}=\frac{(1-r^3)^{m}|\xi|^2}{r^m}
\end{equation}
satisfying  ${\rm Ric} (h_*^m)=m\omega_*$ (cfr. \eqref{rich}).

\begin{prop}\label{propmain}
The \K\ metric $g_*$ on ${\mathcal D}_*$  is strictly  partially regular.
\end{prop}

\begin{proof}
Consider  the   Hilbert space 
$${\mathcal H}_m=\{f\in\Hol({\mathcal D}_*) \ | \ \|f\|_m^2= \int_{{\mathcal D}_*}e^{-m\Phi_*(z)}|f(z)|^2\frac{\omega^2}{2!}<\infty \}.$$ 
We show that 
 ${\mathcal H}_m\neq \{0\}$ for $m\geq 3$ and  an orthonormal basis of ${\mathcal H}_m$
is given by the monomials
$\{\frac{z_1^j z_2^k}{\|z_1^jz_2^k\|_m}\}_{j+k>m-3}$, 
where 
\begin{equation}\label{integral2}
\|z_1^jz_2^k\|^2_m 
=\frac{3j!k!}{4(j+k+1)!}\frac{\Gamma(\frac{j+k-m}{3}+1)(m-3)!}{\Gamma(\frac{j+k-m}{3}+m-1)} \left[ 1 +  \frac{2(\frac{j+k-m}{3}+1)}{\frac{j+k-m}{3}+m-1} \right].
\end{equation}

Let us first see for which values of $j$ and $k$ the monomial $z_1^j z_2^k$ belongs to ${\mathcal H}_m$, namely 
when its norm $\|z_1^jz_2^k\|^2_m$ is finite. 
By passing to polar coordinates $z_1=\rho_1e^{i\theta_1}, z_2=\rho_2e^{i\theta_2}$, $r=\rho_1^2+\rho_2^2$ one gets:
$$\|z_1^jz_2^k\|^2_m=9\int_{{\mathcal D}_*}\frac{(1+2r^3)(1-r^3)^{m-3}}{r^{m-1}}|z_1|^{2j}|z_2|^{2k}(\frac{i}{2\pi})^2dz_1\wedge d\bar z_1\wedge dz_2\wedge d\bar z_2$$
$$=9\int_{r=0}^1\rho_1^{2j+1}\rho_2^{2k+1}\frac{(1+2r^3)(1-r^3)^{m-3}}{r^{m-1}}d\rho_1d\rho_2$$

Now, by setting $r^{\frac{1}{2}}=\rho = \sqrt{\rho_1^2 + \rho_2^2}$ we can make the substitution $\rho_1 = \rho \cos \theta, \rho_2 = \rho \sin \theta$, $0 < \rho < 1$, $0 < \theta < \frac{\pi}{2}$, so that $d \rho_1 d \rho_2 = \rho d \rho d \theta$ and the integral becomes

$$9 \int_{\theta=0}^ {\frac{\pi}{2}} (\cos \theta)^{2j+1} (\sin \theta)^{2k+1} \int_{\rho=0}^1  \rho^{2j+2k-2m+5} (1 +2 \rho^6)(1 - \rho^6)^{m-3} d \rho$$

\begin{equation}\label{integral1}
 =   \frac{9j!k!}{2(j+k+1)!} \int_{\rho=0}^1  \rho^{2j+2k-2m+5} (1 +2 \rho^6)(1 - \rho^6)^{m-3} d \rho
\end{equation}

Let us make the change of variable 

$$x = \rho^{6}, \ \ dx =6 \rho^{5} d \rho$$

and (\ref{integral1}) rewrites

$$\|z_1^jz_2^k\|^2_m = \frac{3j!k!}{4(j+k+1)!} \int_0^{1 } x^{\frac{j+k-m}{3}}(1 +2x)(1 - x)^{m-3} d x.$$

This integral converges if and only if $m\geq 3$ and $j+k>m-3$.
Moreover, formula \eqref{integral2}, easily follows
by using   the well-known fact  that, for any $\alpha, \beta \in \C$ with $Re(\alpha) > 0$, $Re(\beta) > 0$, we have $\int_0^1 x^{\alpha-1}(1-x)^{\beta-1} = \frac{\Gamma(\alpha) \Gamma(\beta)}{\Gamma(\alpha + \beta)}$.

By radiality it follows the monomials $z_1^j z_2^k$ form a complete orthogonal system and hence 
$\{\frac{z_1^j z_2^k}{\|z_1^jz_2^k\|_m}\}_{j+k>m-3}$ turns out to be an orthonormal basis for ${\mathcal H}_m$.

We are now ready to prove  that  $g_*$ is strictly  partially regular  
for $m\geq 3$  with respect to the subspace ${\mathcal S}_m\subset {\mathcal H}_m$ spanned by $\{\frac{z_1^j z_2^k}{\|z_1^jz_2^k\|_m}\}_{j+k-m=3i}$, for $i=0, 1, \dots$.

One needs to verify  that there exists a positive constant $C_m$ (depending on $m$) such that:

\begin{equation}\label{constm}
T_{mg_*}^{{\mathcal S}_m}(z)=\frac{(1-|z|^6)^{m}}{|z|^{2m}}\cdot\sum\limits_{\substack {j+k= m + 3i \\ i = 0, 1, \dots}}
 \frac{|z_1|^{2j} |z_2|^{2k}}{\|z_1^j z_2^k\|^2_m}=C_m,
 \end{equation}
 Equation \eqref{integral2} for $j+k= m + 3i$ becomes
 \begin{equation}\label{integral3}
\|z_1^jz_2^k\|^2_m 
=\frac{3}{4(m-1)(m-2)}{m+3i\choose j}^{-1}{m+i-1\choose i}^{-1}.
\end{equation}

By using 
$$\frac{1}{(1-x)^{m}} = \sum_{i=0}^{\infty} \frac{m(m+1) \cdots (m + i -1)}{i!} x^i=\sum_{i=0}^{\infty} {m+i-1\choose i} x^i$$ 
(and $|z|^2=|z_1|^2+|z_2|^2$) one has

\begin{equation}\label{integral4}
\frac{|z|^{2m}}{(1-|z|^6)^{m}}=
\sum\limits_{\substack {j+k= m + 3i \\ i = 0, 1, \dots}} {m+i-1\choose i}  {m+3i\choose j} |z_1|^{2j} |z_2|^{2k}.
\end{equation}

By combining \eqref{integral3} and \eqref{integral4} one sees that \eqref{constm} is satisfied with 
$C_m=\frac{4}{3}(m-1)(m-2)$, and we are done.
\end{proof}

\begin{remark}\rm
Notice that the metric $g_*$ is radial, namely it admits a \K\ potential depending only on  $|z_1|^2+|z_2|^2$. Moreover, a simple computation shows that $|R_{g_*}|^2-4|{\rm Ric}_{g_*} |^2$ is  a  constant (given by
$-960\pi^2$), where ${\rm Ric}_{g_*}$ and  $R_{g_*}$, are, respectively, the Ricci tensor and the Riemann curvature tensor of the metric $g_*$. 
By using the classification results on radial cscK metrics given in \cite{LMZszego} (see also \cite{LSZ}) one can prove  the following: if $g$ is a radial, strictly partially  regular cscK metric on an $n$-dimensional complex manifold $M$ such that $|R_g|^2-4|{\rm Ric}_g |^2$ is constant,  then $n=2$ and there exist three positive constants $\mu, \lambda$  and $\xi$
such that $\mu$ and $\mu\frac{\lambda}{2}$ are positive integers, 
$M= \left\{ r = |z_1|^2 + |z_2|^2 \ | \ r < \xi^{-\frac{1}{\lambda+1}} \right\}$ and 
$\omega =\frac{i}{2\pi}\partial\bar\partial\Phi(m, \lambda,\xi)$
where 
\begin{equation}
\Phi(m, \lambda,\xi)=m   \log \frac{(|z_1|^2 + |z_2|^2)^{\frac{\lambda}{2}}}{1 - \xi (|z_1|^2 + |z_2|^2)^{(\lambda+1)}}.
\end{equation}
Moreover the scalar curvature $s_g$ of $g$ is given by
\begin{equation}\label{sg}
s_{g}=-\frac{24\pi}{m}.
\end{equation}
Notice that when $m=1$, $\lambda =2$ and $\xi=1$ one regains $M={\mathcal D}_*$ and  $g=g_*$.
\end{remark}

In the proof of Theorem \ref{mainteor} we need   the following:
\begin{lem}\label{lemprod}
Let $(M_1,g_1)$ and $(M_2, g_2)$ be two  K\"ahler manifolds such that $g_1$ is  strictly  partially  balanced and $g_2$ is  (strictly) partially balanced. Then the metric  $g_1\oplus g_2$ on $M_1\times M_2$ is strictly partially balanced.
In particular, if  $g_1$ is strictly partially  regular  and $g_2$ is (strictly) partially  regular
then  $g_1\oplus g_2$  is strictly partially  regular.
\end{lem}
\begin{proof}
For $\alpha=1,2$ let $(L_\alpha, h_\alpha)$ be the Hermitian line bundle over $M_\alpha$ such that ${\rm Ric} (h_\alpha)=\omega_\alpha$ (cfr. (\ref{rich}) in the introduction), where 
$\omega_{\alpha}$ is the \K\ form associated to $g_{\alpha}$. Let $(L_{1, 2}=L_1\otimes L_2, h_{1, 2}=h_1\otimes h_2)$ be the  holomorphic hermitian  line bundle over $M_1\times M_2$ such that ${\rm Ric} (h_{1, 2})=\omega_1\oplus \omega_2$ and 
$$
\mathcal{H}_\alpha=\left\{s\in H^0(L_\alpha)\  |\ \int_{M_\alpha} h_\alpha(s,s)\frac{\omega_\alpha^{n_\alpha}}{n_\alpha!}<\infty\right\},
$$
where $n_\alpha$ is the complex dimension of $M_\alpha$. 
Let ${\mathcal S}_\alpha\subset\mathcal{H}_\alpha$, $\alpha =1, 2$, be the subspace of  $\mathcal{H}_\alpha$ 
with respect to which $g_{\alpha}$ is 
(strictly) partially  balanced. Notice that ${\mathcal S}_1\subsetneq\mathcal{H}_1$ since $g_1$ is 
strictly partially  balanced.

Let $\{s^1_j\}$ (resp. $\{s^2_k\}$) be an orthonormal basis for $\mathcal{S}_1$ (resp. $\mathcal{S}_2$) with respect to the $L^2$-product
induced by $h_1$ (resp. $h_2$).

It is not hard to see (cfr.  \cite[Lemma 7]{LZedda15}) that $\{s^1_j\otimes s^2_k\}$ is an orthonormal basis
for the subspace ${\mathcal S}_{1,2}={\mathcal S}_1\otimes {\mathcal S}_2$  
of the 
Hilbert space
$$\mathcal{H}_{1, 2}=\left\{s\in H^0(L_{1,2}) \  |\ \int_{M_1\times M_2} h_{1, 2}(s,s)\frac{(\omega_1\oplus\omega_2)^{n_1+n_2}}{(n_1+n_2)!}<\infty\right\}.$$
Thus
\begin{equation*}
\begin{split}
T_{g_1\oplus g_2}^{{\mathcal S}_{1,2}}(x, y)=&\sum_{j, k}h_{1,2}(s^1_j(x)\otimes s^2_k(y),s^1_j(x)\otimes s^2_k(y))\\
=&\sum_{j}h_{1}(s^1_j(x),s^1_j(x))\sum_{k}h_{2}(s^2_k(y),s^2_k(y))=T_{g_1}^{{\mathcal S}_{1}}(x)T_{g_2}^{{\mathcal S}_{2}}(y)=C_1C_2.
\end{split}
\end{equation*}
for two positive constant $C_1$ and $C_2$.
Then   $g_1\oplus g_2$ is strictly partially  balanced with respect to the subspace  ${\mathcal S}_{1,2}\subsetneq \mathcal{H}_{1, 2}$.
Assume now that $g_{\alpha}$ is (strictly) partially regular and let  $m_\alpha$ be  such that  $m g_\alpha$ is (strictly) partially  balanced  for $m\geq m_\alpha$. Then, by the first part, $m(g_1\oplus g_2)$ is strictly partially balanced
for $m\geq \max\{m_1, m_2\}$, i.e. $g_1\oplus g_2$ is strictly partially regular. 
\end{proof}

\begin{proof}[Proof of Theorem \ref{mainteor}]
It is well-known (see, e.g. \cite{arlquant}) that the Fubini-Study metric $g_{FS}$ on the complex sphere  $\C P^1$, 
and  the flat metric $g_0$ on the complex Euclidean  space $\C^k$, $k\geq 1$,
are  regular cscK metrics of constant scalar curvatures $s_{g_{FS}}=8 \pi$ and $s_{g_0}=0$ respectively.
By   Proposition \ref{propmain} and  Lemma \ref{lemprod}  the metric $g_*\oplus g_0$  is a strictly partially regular cscK metric with negative scalar curvature  on the complex $n$-dimensional  manifold ${\mathcal D}_*\times\C^{n-2}$, for all $n\geq 2$.

Further, by combining Proposition \ref{propmain}, Lemma \ref{lemprod},  \eqref{sg*}  and \eqref{sg} it follows that  the
\K\ metrics $3g_*\oplus g_{FS}\oplus g_{0}$ and  $4g_*\oplus g_{FS}\oplus g_{0}$ on the complex $n$-dimensional manifold ${\mathcal D}_*\times\CP^1\times \C^{n-3}$, $n\geq 3$,  
are  strictly partially regular metrics with vanishing scalar curvature and postive scalar curvature respectively. 
\end{proof}


\begin{thebibliography}{99}
\bibitem{arlquant} C. Arezzo, A. Loi,
{\em Quantization of K\"{a}hler
manifolds and the asymptotic
expansion of Tian--Yau--Zelditch},
J. Geom. Phys. 47  (2003), 87-99.

\bibitem{arlcomm} C. Arezzo and A. Loi,
{\em Moment maps, scalar curvature and
quantization of K\"{a}hler manifolds},
Comm. Math. Phys. 246 (2004), 543-549.

\bibitem{hombal} C. Arezzo, A. Loi, F. Zuddas,
{\em On homothetic balanced metrics},
Ann.  Global Anal.  Geom. 41, n. 4  (2012), 473-491.

\bibitem{ALZ}
C. Arezzo, A. Loi, F. Zuddas,
{\em Szeg\"{o} Kernel, regular quantizations  and spherical CR-structures}.
Math. Z.  (2013) 275, 1207-1216.

\bibitem{Cal} E. Calabi,  {\em Isometric Imbedding of Complex Manifolds}, Ann. of Math. Vol. $58$ No. $1$,  $1953$.

\bibitem{FCAL} F. Cannas Aghedu, A. Loi,  
{\em The Simanca metric admits a regular quantization}, arXiv:1809.04431.

\bibitem{cgr1} M. Cahen, S. Gutt, J. H. Rawnsley,
{\em Quantization of K\"{a}hler manifolds I: Geometric
interpretation of Berezin's quantization}, JGP. 7 (1990), 45-62.

\bibitem{DHL} A. J. Di Scala, H. Hideyuki, A. Loi,
{\em \K\ immersions of homogeneous \K\ manifolds into complex space forms},
Asian Journal of Mathematics  Vol. 16 No. 3 (2012), 479-488.

\bibitem{donaldson} S. Donaldson, {\em Scalar Curvature and Projective Embeddings, I}, J. Diff. Geom. 59 (2001),  479-522.

\bibitem{loireg} A. Loi,
{\em Regular quantizations of \K\ manifolds and constant scalar curvature metrics},
J. Geom. Phys. 53 (2005), 354-364.

\bibitem{LMZszego}
A. Loi, R. Mossa, F. Zuddas, 
\emph{Finite TYCZ expansions and cscK metrics}, 
arXiv:1903.07679.



\bibitem{LSZ} 
A. Loi, F. Salis, F. Zuddas,
\emph{On the third coefficient of TYZ expansion for radial scalar flat metrics},
J. Geom. Phys. 133,  210-218 (2018).

\bibitem{LZedda11}
A. Loi, M. Zedda, 
{\em \K\--Einstein submanifolds of the infinite dimensional projective space},
Math. Ann.   350 (2011), 145-154.

\bibitem{LZeddaCH}
A. Loi, M. Zedda, 
{\em  Balanced metrics on Cartan and Cartan-Hartogs domains},
Math. Z.  (2012),Vol. 270,  1077-1087.

\bibitem{LZedda15}
A. Loi, M. Zedda, 
{\em  On the coefficients of  TYZ expansion of locally Hermitian symmetric spaces},
Manuscripta Mathematica  (2015),Vol. 148,  303-315.

\bibitem{LoiZedda-book}A. Loi, M. Zedda, \emph{K\"{a}hler Immersions of K\"{a}hler Manifolds into Complex Space Forms}, Lecture Notes of the Unione Matematica Italiana \textbf{23}, Springer, (2018).

 \bibitem{pokross} F.T. Pokorny, J. Ross, \emph{Toric partial density functions and stability of toric varieties}, Math. Ann. 358 (2014), no. 3-4, 879-923.

\bibitem{rawnsley} J. Rawnsley, \emph{Coherent states and K\"ahler manifolds}, Quart. J. Math. Oxford (2), n. 28 (1977), 403--415.

\bibitem{rosssinger} J. Ross, M. Singer \emph{Asymptotics of partial density functions for divisors}.
J. Geom. Anal. 27 (2017), no. 3, 1803-1854.

\bibitem{ZePeng}
S. Zelditch P. Zhou,
\emph{Interface asymptotics of partial Bergman kernels on $S^1$-symmetric \K\ manifolds},
arXiv:1604.06655.

\bibitem{ZePeng2}
S. Zelditch P. Zhou,
\emph{Central Limit theorem for spectral Partial Bergman kernels},
arXiv:1708.09267.

\bibitem{ZePeng3}
S. Zelditch P. Zhou,
\emph{Interface asymptotics of partial Bergman kernels around a critical level},
 arXiv:1805.01804.
 
 \bibitem{ZePeng4}
S. Zelditch P. Zhou,
\emph{Pointwise Weyl laws for Partial Bergman kernels},
  arXiv:1805.05203.



\bibitem {tian}
G. Tian, 
\emph{On a set of polarized \K\ metrics on algebraic manifolds},
J. Diff. Geom. 32, 99-130 (1990).


\end{thebibliography}
\end{document}